\documentclass[11pt,a4paper,oneside,draft]{amsart}
\usepackage[T1]{fontenc}
\usepackage{amsthm}
\usepackage[leqno]{amsmath}
\usepackage{amssymb}
\usepackage{amsfonts}
\usepackage{fancyhdr}
\usepackage{color}
\usepackage{tikz-cd}
\usepackage[utf8]{inputenc}
\usepackage{mathrsfs}
\usepackage[utf8]{inputenc}
\usepackage{blindtext}
\usepackage[a4paper, total={6in, 8in}]{geometry}

\usepackage{enumitem}

\usetikzlibrary{decorations.pathmorphing}

\newtheorem{thm}{Theorem}

\newtheorem{lemat}[thm]{Lemma}

\newtheorem{cor}[thm]{Corollary}
\newtheorem{prop}[thm]{Proposition}
\newtheorem{fact}[thm]{Fact}
\theoremstyle{definition}
\newtheorem{defi}[thm]{Definition}

\numberwithin{thm}{section}
\title[Type III contractions and quintic threefolds]{Type III contractions and quintic threefolds}
\author{Kacper Grzelakowski}
\begin{document}
\begin{abstract}
We study type III contractions of Calabi-Yau threefolds containing a ruled surface over a smooth curve. We discuss the conditions necessary for the image threefold to by smoothable. We describe the change in Hodge numbers caused by this contraction and smoothing deformation. A generalization of a fomula for calculating Hodge numbers of hypersurfaces in $\mathbb{P}^4$ with ordinary double and triple points is presented. We use these results to construct new Calabi-Yau threefolds of Picard rank two arising from a family of quintic threefolds containing a cone.
\end{abstract}
\maketitle

By a Calabi-Yau threefold we mean a complex projective threefold $X$ with $K_X=0$ and $h^1(\mathscr{O}_X)=h^2(\mathscr{O}_X)=0$. We allow $X$ to have some singularities, namely ordinary double and ordinary triple points. A primitive contraction is a birational morphism between Calabi-Yau threefolds decreasing the Picard rank by one. Contraction is of type I if it contracts a finite set of curves to points, of type II if it contracts a divisor to a point and of a type III if it contracts a divisor to a curve. It is known that certain deformation classes of Calabi-Yau threefolds are linked with each other by so called extremal transitions, that is primitive contractions followed by smoothing. These transitions are called conifold if the singularities in the image of the contraction are ordinary double points.  In this paper we aim to discuss the process of obtaining new Calabi-Yau threefolds through a conifold transition involving a type III primitive contraction. In particular we construct a family of quintic threefolds containing a cone which give rise to new Calabi-Yau threefolds of Picard rank 2.

Our results are in the spirit of the so called Reid's fantasy. M. Reid in \cite{MiR} has conjectured that there could exist an irreducible space of Calabi-Yau threefolds such that any Calabi Yau threefold would be a small resolution of a degeneration of this family to something with ordinary double points. We can think of this space as of a graph with each node representing one deformation class. Finding paths between these nodes is an actively studied problem. Wilson in \cite{PMHW1} provided a detailed description of the K\"{a}hler cones of Calabi-Yau threefolds showing in particular that codimension one faces of their cones correspond to primitive contractions. In \cite{MG1} and \cite{MG2} Gross described the conditions for the primitive Calabi-Yau threefolds obtained by such contraction to be smoothable thus giving a link between nodes of the graph.

It is known that a type III contraction of a divisor $E$ to a curve $C$ deforms to a type I contraction providing $g(C) > 1$ \cite{MG1}. It is conjectured that sufficiently general transition arising from the type III contraction deforms to a conifold transition and we show that in the case of surfaces ruled over smooth curves this conjecture holds. We provide the formula describing the change in Hodge numbers of threefolds under this transition. These results are in line with those obtained in \cite[3.4]{KMP}. 
Namely we show 
\begin{prop}
Let $X$ be a smooth Calabi-Yau threefold containing a smooth surface $E$ ruled over a smooth curve $C$ of genus $p_a(C)>1$, let $\pi:X\to Y$ be a primitive type III contraction and let $\tilde{Y}$ be the smooth Calabi-Yau obtained by deforming $Y$. Then
\begin{enumerate}[resume]
\item $h^{1,1}(\tilde{Y})=h^{1,1}(X)-1$
\item $h^{1,2}(\tilde{Y})=h^{1,2}(X)+2p_a(C)-3.$
\end{enumerate}
\end{prop}

To provide examples of type III smoothable contractions we study the geometry of quintic threefolds containing a cone over a curve. We choose quintic threefolds as they are natural objects being hypersurfaces of a relatively low degree and dimension yet having interesting properties. Geometry of quintic threefolds has been studied for example in  \cite{BVG} or \cite{CHM}. Nodal quintic threefolds have been analysed by Friedman \cite{FR} and van Straten \cite{DVS}, while the problem of a possible number of triple points on such hypersurface has been the subject of recent work by Kloosterman and Rams \cite{KR}. We hope that our results may find  application in solving that last problem. 

We resolve singularities of quintics with triple point at the vertex of a cone to obtain smooth threefolds containing a ruled surface and use type III contractions to construct new Calabi-Yau threefolds with Picard rank $2$. We provide a generalization of the results of Cynk from \cite{CS4} and \cite{CS3} regarding the Hodge numbers of hypersurfaces in $\mathbb{P}^4$ with ordinary double and triple points and use it to calculate the Hodge numbers of quintics in question. In our analysis we also adapt some of the results of Kapustka and Kapustka from \cite{GK} and \cite{GK2} where they discuss primitive type II contractions of Calabi-Yau threefolds, describe their images and smoothing families and give the formula for calculating Hodge numbers of threefolds obtained through this process. 

The organization of the paper is as follows. In Section 1 we recall the general facts concerning the quintic threefolds with triple points. We proceed with analysing the geometry of those quintics that contain a cone over a smooth curve. We discuss the intersection theory on these threefolds and describe how their K\"{a}hler cone looks like providing they are of Picard rank $3$. We show that such Calabi-Yau threefolds admit a type III contraction and either one type I and one type II contraction or two different type I contractions. In Section 2 we give the formula to calculate the Hodge numbers of hypersurfaces in $\mathbb{P}^4$ with ordinary double and triple points. We also provide the general formula for calculating Hodge numbers of Calabi-Yau threefolds obtained after type III contraction of a ruled surface to a curve with $g(C)>1$ and a smoothing deformation.

Section 3 is devoted to giving examples of the above procedure. We discuss in detail the construction of a family of quintic threefolds containing a cone over a curve and the resolution of their singularities which produces threefolds containing a smooth ruled surface. Then we describe the contractions and smoothings of the latter and finally we use the formulas from the previous section to calculate the Hodge numbers of obtained threefolds. The procedure we describe is illustrated on the diagram below. Here $\bar{X}$ is a quintic threefold containing a cone over a smooth curve with triple point at the vertex of the cone and possible nodes at its surface, $\tilde{X}$ and $X$ denote threefolds obtained after the small resolution of nodes and the consecutive blow up of the triple point, $Y$ is the image of the type III contraction of $X$ and $\tilde{Y}$ is a smooth deformation of $Y$.   
\begin{center}
\begin{tikzcd}
       &                           & X \arrow[ld] \arrow[rd, "\pi"]          &                  &          \\
\bar X & \tilde X \arrow[l] & {} & Y\in\mathscr{Y} \ni \tilde Y 
\end{tikzcd}
\end{center}
 Tables containing numerical results are presented at the end of the paper.
\section*{Acknowledgements}
The author wishes to express his deepest gratitude for the continuous help and support to G. Kapustka and to S. Cynk for the helpful discussion of his results. Author is supported by the project Narodowe Centrum Nauki 2018/30/E/ST1/00530.

\section{Geometry of quintic threefolds containing a cone with triple point at its vertex}
\subsection{Quintic threefolds with triple point}The object of our particular interest is a quintic hypersurface $X$ in $\mathbb{P}^4$ admitting an ordinary triple point. If we assume that the triple point is $O=[0:0:0:0:1]$ the defining equation $F$ of $X$ can be written in the form $F=u^3F_3(x,y,z,t)+uF_4(x,y,z,t)+F_5(x,y,z,t)$ where $u$ is the last coordinate of $\mathbb{P}^4$ and $F_i$ are homogeneous polynomials in given variables defining smooth surfaces in $\mathbb{P}^3$. In the latter we will often skip the variables and just write $F_i$ if the context is clear. For $O$ to be an ordinary triple point of $X$ we need $V(F_3)$ to define a smooth cubic surface. 
\begin{lemat}
Let $X$ be a Calabi-Yau threefold admitting only ordinary triple points as singularities. Then $\tilde{X}$ obtained by blowing up the singular locus of $X$ is again a Calabi-Yau threefold.  
\end{lemat}
\begin{proof}
This standard fact is a consequence of a resolution $\pi:\tilde{X}\to X$ of an ordinary triple point being crepant, that is satisfying $K_{\tilde{X}}=\pi^{*}K_X$.
\end{proof}
When we blow up $X$ in $O$ we replace the ordinary triple point with a smooth cubic surface $D$ which is the intersection of $\tilde{X}$ with the exceptional $\mathbb{P}^3$. We recall basic facts about $D$ which we use in the following. For the reference see \cite[V 4.7]{RH}. We think of $D$ as a blow-up of $\mathbb{P}^2$ in six points $P_1,\dots, P_6$ in general position. Let us denote by $h$ the pullback of a hyperplane section of $\mathbb{P}^2$ under this blow up and by $e_i$'s six exceptional divisors over $P_i$'s. Then we have a standard
\begin{lemat}
Picard group $Pic(D)$ of a cubic surface is isomorphic to $\mathbb{Z}^7$. It is generated by $h$ and $e_i$'s. Each curve $C$ on $D$ can be written as $c=ah-\Sigma_{i=1}^6b_ie_i$, for $a,b_i\in \mathbb{Z}$.
\end{lemat}

We narrow our focus to the specific family of quintic threefolds in $\mathbb{P}^4$ namely those containing a cone $\bar{E}$ over a smooth curve $C$ with $O$ being the vertex of this cone. We aim to describe some aspects of the geometry of their resolutions, their Picard group and their K\"{a}hler cone. In the last section of this paper we will describe in detail the family of quintic threefolds we have managed to construct but for now we proceed with more general statements. 

We begin with a Calabi-Yau quintic hypersurface in $\mathbb{P}^4$ containing a cone $\bar{E}$ over a smooth curve $C$ with vertex of a cone $O$ being a triple point of $\bar{X}$. Note that in this case necessarily we have $C\subset V(F_3)\cap V(F_4)\cap V(F_5)\subset\mathbb{P}^3$ when we write $X=V(F)$ as before. It may happen that apart from $C$ there are isolated points in the intersection $V(F_3)\cap V(F_4)\cap V(F_5)$, we call them excess points. They give rise to lines passing through the triple point of $\bar{X}$ and are otherwise disjoint from $\bar{E}$. Furthermore, threefold $\bar{X}$ can have other singular points than the vertex of the cone but we show that they are ordinary double points admitting a small resolution. To that aim let $\Lambda$ be a system of quintic threefolds containing a cone $\bar{E}$ over a given curve $C$ with triple point at the vertex of the cone. We have

\begin{thm}
For any nonsingular curve $C$ a general element $\bar{X}$ of $\Lambda$  outside of the vertex of the cone has at worst nodes as singularities. Furthermore, the singular points of $\bar{X}$ lie on a cone. 
\end{thm}
\begin{proof}
Let $\pi:\tilde{\mathbb{P}^4}\to\mathbb{P}^4$ be the blow-up of the $\mathbb{P}^4$ along $\bar{E}$ and let $E$ be the exceptional divisor. $\bar{E}$ is the base locus of $\Lambda$ -  the system of quintic threefolds containing $\bar{E}$ and thus $\pi^*\Lambda-E$ is base-point free. 

Consider the morphism $\Phi:\tilde{\mathbb{P}^4}\to \mathbb{P}^N$ given by the linear system $\pi^*\Lambda-E$. For each point $p$ of $\bar{E}$ we obtain that $\Phi$ embeds $\pi^{-1}(p)$  as a line $\mathbb{P}^1$ in $\mathbb{P}^N$. Let $\mathbb{P}^{N^{*}}$ be the dual projective space to $\mathbb{P}^{N}$. We work in the product $\bar{E}\times  \mathbb{P}^{N^{*}}$ with natural projections $\pi_1$,$\pi_2$. We define $I=\{(p,H):\pi^{-1}(p)\subset H\}$. It follows that $\pi_1^{-1}(p)$ is the set of hyperplanes containing a fixed line so the dimension of $\pi_1^{-1}(p)$ is $N-2$ which says that the set of quintics having point $p$ as a singularity is of codimension $2$. Furthermore we get $\dim I=\dim\bar{E}+N-2=2+N-2=N$ and thus $\dim\pi_2^{-1}(H)=0$ for general $H$. Since $\pi_1(\pi_2^{-1}(H))$ is the singular locus of the general $\bar{X}$ we have that the singular locus is at most zero-dimensional.  

At every point of the cone $\bar{E}$ which is not the vertex we can find quintics which have different tangent directions. To see that observe that if $\bar{X}=V(F)$ is the quintic containing $\bar{E}$ and $P$ is the point on $\bar{E}$ we have $F(P)=0$ where $F=u^2F_3+uF_4+F_5$ and $F_i$ are homogeneous polynomials of degree $i$ as above. We can substitute $\tilde{F}_4=F_4+F_1F_3$ and $\tilde{F}_5=F_5+F_2F_3$ where $F_1$ and $F_2$ are general linear and quadratic homogeneous polynomials thus obtaining equation of another quintic containing the cone $\bar{E}$ but having different tangent direction at $P$. For any $\alpha\in\{x,y,z,t\}$ we obtain $$\frac{\partial \tilde{F}}{\partial \alpha}=\frac{\partial F}{\partial \alpha}+ F_3(u\frac{\partial F_1}{\partial \alpha}+\frac{\partial F_2}{\partial \alpha})+\frac{\partial F_3}{\partial \alpha}(uF_1+F_2).$$ Since $F_3$ is the equation of the smooth cubic surface it is enough that we have $F_2(P)\neq -uF_1(P)$ for the statement to be true. This construction also shows that for each point $P\in \bar{E}$ there exists a quintic containing $\bar{E}$ which is not singular at $P$. By Bertini theorem we conclude that general $\bar{X}$ has at worst double points as singularities but for the vertex of the cone. What is more the singularities lie on the cone $\bar{E}$ and outside of the cone general $\bar{X}$ is smooth.

By following the argument of \cite [Theorem 4.4] {KJ} we conclude that a generic element $\bar{X}$ of $\Lambda$ has only singularities of $cA$ type as for every point $P$ of $\bar{E}$ but for the vertex we can find $\bar{X}_P$ which is not singular at $P$. From \cite[Claim 2.2]{DH},  we have that $P\in \bar{E}$ is a singular point of a general $X\in\Lambda$ if and only if $X'$ such that $\pi(X')=X$ contains the whole fiber $\pi^{-1}(P)$.

Now let $P$ be a singular point (different than the vertex of the cone) of a general quintic $\bar{X}$ containing $\bar{E}$. We blow up the whole $\bar{E}$ and obtain $\mathbb{P}^1$ as a fiber of $\pi$ over $P$. From this and the previous paragraph we see that a general element $X$ of $\Lambda$ admits a small resolution (outside $O$) with $\mathbb{P}^1$ as an exceptional divisor. Using Bertini we see that a general element of $|\pi^*\Lambda-E|$ cuts $E$ along a nonsingular surface and so a normal bundle of $C$ contains a subbundle $\mathscr{O}_C(-1)$. Following the argument from \cite{GK2} Theorem 2.1 we know that in a small resolution of a $cA$ type singularity there is a curve with normal bundle  $\mathscr{O}(-1)\bigoplus\mathscr{O}(-1)$ or $\mathscr{O}\bigoplus\mathscr{O}(2)$. In our case this has to be the former of the two which means the singularities of a generic $\bar{X}$ outside the vertex of the cone are ordinary double points. 

\end{proof}

Our aim is to obtain smooth $X$ and we do so by first blowing up the point $O$  and then performing a small resolution of nodes as ilustrated in the following diagram.
\begin{center}
\begin{tikzcd}
E \arrow[d] \arrow[r, hook]             & X \arrow[d, "\tilde{\pi}"]       \\
\tilde{E} \arrow[d] \arrow[r, hook] & \tilde{X} \arrow[d, "\bar{\pi}"] \\
\bar{E} \arrow[r, hook]               & \bar{X}                 
\end{tikzcd}
\end{center}
Here $\bar{\pi}:\tilde{X}\to X$ is the blow up of $O$ and $\tilde{\pi}: X\to\tilde{X}$ is the subsequent small resolution. In what follows we will use $\pi:X\to\bar{X}$ to denote the composition $\tilde{\pi}\circ\bar{\pi}$.
\subsection{Intersection theory of $X$}

Let $H_X$ be the strict transform of a general hyperplane section of $\bar{X}$ and $H_X'$ the strict transform of the hyperplane section of $\bar{X}$ passing through $O$. Let $E$ be the strict transform of the cone $\bar{E}$ over the curve $C$ on $\bar{X}$ and let $D=\pi^{-1}(O)$ be the exceptional divisor of $\pi$, that is a smooth cubic surface contained in $X$. We denote by $l$ the fibre of $E$. We write $g$ for genus of $C$.
\begin{thm}\label{TPG}
If the Picard group of $X$ is of rank $3$ then it is generated by the divisors $H_X$, $D$, and $E$.
\end{thm}
We limit ourselves to the case when $\text{rank}$ of the Picard group of $X$ is $3$ as this is the case for all the threefolds we construct in the last section. Also, we are inclined to believe this should always be the case when $X$ is the resolution of $\bar{X}$ containing exactly one cone and one ordinary triple point being its vertex. To prove the theorem we need to calculate the intersections of divisors with themselves and with curves on $X$. We provide these in the following lemmas.
\begin{lemat}
On $X$ we have
\begin{enumerate}
\item $H_X.D=0$
\item $H_X.E=C+deg(C)l$
\item $D.E=C$
\item $E^2=K_E=-2C+(K_C-deg(C))l$
\item $E^3=K_E^2=8(1-g)$
\item $D^2=K_D$
\item $D^3=3$
\end{enumerate}
\end{lemat}
\begin{proof}
General hyperplane section of $\bar{X}$ misses the point $O$ and thus $H_X$ misses $D$ by the properties of the pullback which proves (1). The exceptional divisor $D$ intersected with $E$ is precisely the base curve $C$  \cite[V.2.11.4]{RH} and thus $D.E=C$ hence (2). Hyperplane passing through $O$ cuts $\bar{E}$ in exactly $deg(\bar{E})=deg(C)$ fibers. By \cite[Corollary 9.12] {EH}, we can think of the class $D$ as a class $H_X-H_X'$ thus $H_X\sim D+H_X'$ and so $H_X.E=(D+H_X').E=C+deg(C)l$. We calculate $E^2$ from the adjunction formula. $K_E=(K_X+E)|_E=E^2$ and similarily $K_E^2=(K_X+E)^2|_E=(K_X^2+2K_XE+E^2)|_E=E^2|_E=E^3$  since $K_X=0$. Again by \cite{RH} we know that $E^3=8(1-g)$. Similarily from adjunction $D^2=K_D$ and since $D$ is a smooth cubic surface $K_D\sim -H_D$ where $H_D$ is a hyperplane section of a cubic. This is a standard use of adjunction as $K_D\sim \mathscr{O}_D(-3-1+deg(D))=\mathscr{O}_D(-1)\sim -H_D$. Since $D^3\sim K_D^2$ we obtain $D^3\sim(-H_D)^2=3$ which completes the proof.
\end{proof}

\begin{lemat}
We have 
\begin{enumerate}
\item $l^2=0$
\item $C^2=-deg(C)$
\item $l.C=1$
\item $l.E=-2$
\item $C.E=K_C+deg(C)$
\item $D.C=-deg(C)$
\end{enumerate}
\end{lemat}
\begin{proof}
Points (1)-(4) are standard facts for ruled surfaces and are for example found in \cite[V.2]{RH}. To obtain (5) we can restrict $C.E_{|E}$ and thus $C.K_E=C.(-2C+(K_C-deg(C)l)=2deg(C)+K_C-deg(C)=K_C+deg(C).$ By the similar argument $D.C=K_D.C=-H_D.C=-deg(C)$.
\end{proof}
We can proceed with
 \begin{proof}[Proof of \ref{TPG}]
We need to show that $H_X, D$ and  $E$ are lineary independent in $Pic_{\mathbb{Q}}(X)$. To this end assume that $T=\alpha H_X+\beta D+\gamma E=0$ for $\alpha,\beta,\gamma\in\mathbb{Q}$. But then also $T.H_X^2=0$ and so $5\alpha+deg(C)\gamma =0$ thus $\alpha=-\frac{1}{5}deg(C)\gamma$. As $T.H_X=0$ we obtain $\alpha H_X^2+\gamma (C+deg(C)l)=0$ and so $\gamma (-\frac{1}{5}deg(C)H^2_X+C+deg(C)l)=0$. If $\gamma\neq 0$ we obtain $H^2_X=5\frac{C+deg(C)l}{deg(C)}$. Since $deg(H^2_X)=5$ this cannot happen. Thus $\gamma=0$, hence $\alpha=0$ and $\beta=0$ making $H_X,D$ and $E$ linearly independent as desired. 
\end{proof}

Let $r$ denote the class of $\mathbb{P}^1$ coming from the small blow-up of a node of $X$ lying on the ruled surface. Recall that $V(F_3)\cap V(F_4) \cap V(F_5)$ contains $C$ and some excess points, thus $\bar{X}$ contains a cone $\bar{E}$ over $C$ and lines over excess points. We use $t$ to denote the class of a strict transform of these lines in $X$. 
\begin{lemat}
On $X$ as above we have
\begin{enumerate}
\item $H_X.r=0$
\item $D.r=0$
\item $E.r=1$
\end{enumerate}
Furthermore, we have numerical equivalence of curves $t\sim l+2r$.
\end{lemat}
\begin{proof}
As $r$ is the $\mathbb{P}^1$ replacing the point lying directly on $\bar{E}$ where $\bar{E}$ is smooth the equality $E.r=1$ is clear. General hyperplane section of $\bar{X}$ misses this point giving $H_X.r=0$ and since $D$ is the blow-up of the vertex of $\bar{E}$ it misses the $r$ as well. One can easily derive the numerical equivalence class of $t$ by looking at the Table \ref{tabint} describing the intersection of curves and divisors on $X$.
\end{proof}

We need to consider the curves on $X$ that lie on the cubic surface $D$. Since $H_{X|D}=0$, $D_{|D}=-H_D$, $E_{|D}=C$ we see that restriction of $Pic(X)$ to $Pic(D)$ is at most two dimensional (it may happen that curve $C$ is the multiple of hyperplane section of $D$ and then it is one dimensional). We describe the intersection of curves on $D$ with divisors of $X$. We have $H_D=3h+\Sigma e_i$ and $C=ah+\Sigma b_ie_i$ for $a, b_i\in \mathbb{Z}$. Then we show

\begin{lemat}
For $h$, and $e_i$, $i\in 1,\dots,6$ generators of $Pic(D)$ we have the following equivalences
\begin{enumerate}
\item $h\sim \frac{-3}{-3a+\Sigma b_i}C+(a+\frac{3(a^2-\Sigma b_i^2)}{-3a+\Sigma b_i})r$
\item $e_i\sim \frac{1}{3}C+(b-\frac{a^2-\Sigma b_i^2}{3})r.$
\end{enumerate}
\end{lemat}
\begin{proof}
Again we use Table \ref{tabint} to deduce the equivalences.
The equivalence classes of $h$ and $e_i$ are well defined as long as $3a-\Sigma b_i\neq 0$ but this expression is exactly $deg(C)$ by \cite{RH} and so never $0$.
\end{proof}

\subsection{K\"{a}hler cone of X}
We keep the notation of the previous section. We proceed with the description of the K\"{a}hler cone $K$ of $X$. Recall that the K\"{a}hler cone of a manifold $X$ is spanned in $Pic(X)$ by ample divisors, that is divisors $T$ such that $T.c>0$ for each effective curve $c$ on $X$. The closure $\bar{K}$ of $K$ includes also nef divisors that is divisors $T$ for which $T.c\geq 0$. By the Fact 1 in \cite{PMHW1} divisors that lie on codimension one faces of $\bar{K}$ give primitive contractions of Calabi-Yau threefolds, providing they do not belong to the cubic cone that is $W=\{ T\in Pic_{\mathbb{R}}(X):T^3=0\}$. 
For simplicity we write $H$ instead of $H_X$ to denote the pullback of the hyperplane section of $\bar{X}$ on $X$.
We claim the following:
\begin{thm}\label{KC}
Let $X$ be as above with Picard group generated by $H, D$ and $E$. Then the closure $\bar{K}$ of the K\"{a}hler cone $K$ is a convex hull of three rays. Two of the rays are generated by divisors $H$ and $H-D$ and don't lie on the cubic cone $W$.
\end{thm}

\begin{proof}
First note $H$ and $H-D$ are nef. It is enough to check the intersections with curves $l,c$ and $r$ to see they are all non-negative. Now observe that $H^3=5>0$ and $(H-D)^3=H^3+3H^2.D+3H.D^2+D^3=5+3=8>0$ so indeed these divisors do not lie on the cubic cone $W$. If we write divisors in $Pic(X)$ as $\alpha H+\beta D+\gamma E$ we see that any divisor such that $\gamma=0$ has zero intersection with $r$ and divisors $H$ and $H-D$ satisfy this condition. Since $H.D=H.C=H.r=0$ and $H-D.l=H-D.r=0$ we indeed have each of them contracting something else then $r$ an thus they have to lie on two sides of the cone. 

Let us analyse the side of the cone on which there are divisors $Q$ such that $Q.l=0$. In particular $H-D$ lies on this side. It is straightforward to see that divisors $Q$ have to be of the form $(2\gamma-\beta)H+\beta D+\gamma E$ for $\beta, \gamma\in\mathbb{R}$ (with some additional conditions on $\beta$ and $\gamma$ to maintain nefness which we do not check for the moment). When $\gamma=0$ we recover $H-D$ or a multiple of thereof and so we can focus on the case when $\gamma\neq 0$. Since we are working on a cone we can assume $\gamma=1$ (it cannot be negative as then $Q.r<0$) and so $Q$ is of the form $(2-\beta)H+\beta D+E$. As each such $Q$ has zero intersection with $l$ and a positive intersection with $r$  we can only hope for it to contract some other curve on $D$, be it $C$ or something else. We use the Corollary 4.13 from \cite{RH} to find the value of $\beta$ for which $Q_{|D}$ stops being ample and denote this divisor $L$. We know $Q_{|D}\sim ah+\Sigma b_ie_i$ is ample on $D$ if and only if $b_i>0$, $a>b_i+b_j$, and $2a>\Sigma_{i\neq j}b_i$ for all $i,j$. Note that $Q_{|D}=-\beta H_D+C$ and so for any given $C$ the calculations are straightforward. It turns out divisor $L$ is exactly the third ray spanning the K\"{a}hler cone of $X$. Since $H$ has zero intersection with the whole $D$ any linear combination of $H$ and $L$ has to be zero on some set of curves on $D$ (or $D$ itself) and thus span the codimension one side of the K\"{a}hler cone. 
\end{proof}
\begin{center}
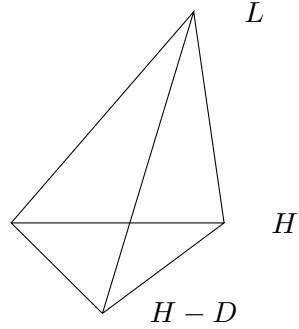
\begin{figure}
\begin{tikzpicture}[scale=.8]
\draw (0,0)--(3.5,0)--(1.5,-1.5)--cycle;
\draw (0,0)--(3,3.5)--(1.5,-1.5) (3,3.5)--(3.5,0);
\node at (3,-1.5){$H-D$};
\node at (4.5,0){$H$};
\node at (4,3.5){$L$};
\end{tikzpicture}
\caption{K\"{a}hler cone of $X$}
\end{figure}
\end{center}

\section{Hodge numbers}
\subsection{Hodge numbers of $X$}
In the following discussion we follow closely the reasoning of \cite{CS3}. The main idea is to combine the results of \cite[Theorem 10] {CS2}, and \cite[Theorem 5]{CS3}, regarding the hypersurfaces in $\mathbb{P}^4$ with only ordinary double points as singularities and hypersurfaces in $\mathbb{P}^4$ with only ordinary triple points as singularities. We use the notion of the equisingular ideal (e.g. from \cite{CS3}). Adapted to our situation it looks as follows:
$$I_{eq}=\bigcap_{i=1}^{\mu}(m_i+\text{Jac} F)\cap(m_0^3+\text{Jac} F)$$
where $m_i$ is the maximal ideal of  a point $P_i$.
We denote $S=\mathbb{C}[X_0,X_1,X_2,X_3,X_4]$ - the polynomial ring in $5$ variables. By $I^{(n)}=I\cap S^n$ we denote the degree $n$ summand of an ideal in a graded ring $S$. 
We use $\delta$ to denote the defect of $X$ that is the difference between the expected and the actual dimension of an ideal of $2d-5$ degree ($5$ in our case) forms vanishing in the singular locus of $X$. Thus $\delta=\dim I_{eq}^{(5)} - ({{2d-1}\choose{4}} -11\mu_3-\mu_2)$ where $\mu_i$ is the number of points of multiplicity $i$ on $\bar{X}$ for $i\in\{2,3\}$ (note that $\mu_3=1$ in our case). We dedicate this section to showing the following

\begin{thm}\label{CHN}
Let $X$ be a threefold obtained from resolution of a quintic threefold $\bar{X}$ with ordinary double and triple points as only singularities (resolution being small for double points). Then:
\begin{enumerate}
\item $h^{0,0}(X)=h^{3,3}(X)=h^{3,0}(X)=1$
\item $h^{1,0}(X)=0$
\item$ h^{1,1}(X)=1+\mu_3+\delta$ 
\item $h^{1,2}(X)= {{2d-1}\choose{4}} -5{ {d}\choose{4}}-11\mu_{3} -\mu_{2}+\delta$ 
\end{enumerate}
\end{thm}
First two identities are standard for quintic threefolds, the last two we show after the preparatory discussion. 
Let $P_0$ be the triple point while $P_1,\dots, P_{\mu}$ are the nodes of $\bar{X}$. Denote $P=\text{Sing}(\bar{X})=\Sigma_{i=0}^{\mu} P_i$, $P'=P-P_0$. Analogously $E_i=\pi^{*}(P_i)$ and $E=\Sigma_{i=0}^{\mu} E_i$. Also denote $E'=E-E_0$. 
From the properties of the blow-up we have
$$K_{\mathbb{\tilde{P}}^4}\sim\pi^{*}K_{\mathbb{P}^4}+3E$$
$${X}\sim\pi^{*}\bar{X}-2E'-3E_0$$
Here we recall Proposition 1. of \cite{CS3} with slightly adapted point $(6)$. 
\begin{prop}[Proposition 1, \cite{CS3}]
\
\begin{enumerate}
\item $\pi_{*}\mathscr{O}_{\tilde{\mathbb{P}}^4}(-mE)\cong\mathscr{J}_{mP}$, for $m\geq 0$
\item $R^i\pi_{*}\mathscr{O}_{\tilde{\mathbb{P}}^4}(-mE)=0$, for $i\neq 0, m\geq 0$
\item $H^i(\mathscr{O}_X)=0$, for $i=1,2$
\item $H^1(\Omega^3_{\mathbb{P}^4})=0$, for $i\leq 2$.
\item $H^i(\Omega^4_{\tilde{\mathbb{P}}^4}(X)\cong H^i(\Omega^4_{{\mathbb{P}}^4}(\bar{X}))$
\item $H^i(\Omega^4_{\tilde{\mathbb{P}}^4}(2X)\cong H^i\Omega^4_{\mathbb{P}^4}((2\bar{X})\otimes\mathscr{J}_{P'+3 P_0})$
\end{enumerate}
\end{prop}
\begin{proof}
We only need to show (6). As in \cite{CS2} we consider
$$K_{\mathbb{\tilde{P}}^4}+2X\sim\pi^{*}(K_{\mathbb{P}^4}+2\bar{X})-E'-3E_0$$
and thus from the projection formula we obtain
$$H^0(\Omega^4_{\tilde{\mathbb{P}}^4}(2{X})=H^0(\Omega^4_{\mathbb{P}^4}(2\bar{X})\otimes\mathscr{J}_{P'+3 P_0}).$$
\end{proof}

\begin{lemat}
The following sequence is exact
$$H^0\Omega^3_{\mathbb{P}^4}(\bar{X})\longrightarrow H^0(\Omega_{\mathbb{P}^4}^3(\bar{X})\otimes\mathscr{O}_{P_0})\longrightarrow H^1(\Omega^3_{\tilde{\mathbb{P}}^4}(X))\longrightarrow 0$$
\end{lemat}
\begin{proof}
By \cite{CS3} and \cite{CS4} we have 
$$\pi^{*}\Omega^3_{\mathbb{P}^4}\cong\Omega^3_{\tilde{\mathbb{P}}^4}(\log E)(-3E)$$
and 
$$\mathscr{O}_{\tilde{\mathbb{P}}^4}(X)\cong\pi^{*}(\mathscr{O}_{\mathbb{P}^4}(\bar{X}))\otimes\mathscr{O}_{\tilde{\mathbb{P}}^4}(-2E'-3E_0).$$

By \cite[6, 2.3(c)]{EV} , setting $D=D_1=E$ we get the following exact sequence
$$0\longrightarrow\Omega^3_{\tilde{\mathbb{P}}^4}(\log E)(-E)\longrightarrow \Omega^3_{\tilde{\mathbb{P}^4}}\longrightarrow\Omega^3_{E|E}\longrightarrow 0$$
which translates into
$$0\longrightarrow\pi^{*}\Omega^3_{{\mathbb{P}}^4}(2E)\longrightarrow \Omega^3_{\tilde{\mathbb{P}^4}}\longrightarrow\Omega^3_{E|E}\longrightarrow 0$$
and tensoring with $\mathscr{O}_{\tilde{\mathbb{P}}^4}(X)$ we get
$$ 0\longrightarrow\pi^{*}\Omega^3_{\tilde{\mathbb{P}}^4}(\bar{X})\otimes\mathscr{O}_{\tilde{\mathbb{P}^{4}}}(-E_0)\longrightarrow \Omega^3_{\tilde{\mathbb{P}^4}}(X)\longrightarrow\Omega^3_{E_0}(3)\otimes\Omega^3_{E'}(2)\longrightarrow 0$$
Applying the direct image we obtain $\pi_{*}\Omega^3_{\tilde{\mathbb{P}^4}}(\tilde{X})=\Omega^3_{{\mathbb{P}^4}}(X)\otimes\mathscr{J}_{P_0}$.Again by \cite{CS3}
$$H^i\Omega_{\tilde{\mathbb{P}}^4}^3(X)=H^i(\Omega^3_{\mathbb{P}^4}(\bar{X})\otimes\mathscr{J}_{P_0})$$
and
$$R^i\pi_{*}\Omega_{\tilde{\mathbb{P}}^4}(X)=0$$
for $i>0$.
Long exact sequence coming from 
$$0\longrightarrow \Omega^3_{\mathbb{P}^4}(\bar{X})\otimes\mathscr{J}_{P_0}\longrightarrow\Omega^3_{\mathbb{P}^4}(\bar{X})\longrightarrow\Omega^3_{\mathbb{P}^4}(\bar{X})\otimes\mathscr{O}_{P_0}\longrightarrow 0$$
gives us 
$$H^0(\Omega^3_{\mathbb{P}^4}(\bar{X}))\longrightarrow H^0(\Omega^3_{\mathbb{P}^4}(\bar{X})\otimes\mathscr{O}_{P_0})\longrightarrow  H^1(\Omega^3_{\mathbb{P}^4}(\bar{X})\otimes\mathscr{J}_{P_0})\longrightarrow H^1(\Omega^3_{\mathbb{P}^4}(\bar{X}))=0$$
which finishes the proof.
\end{proof}

\begin{lemat}
We have the following exact sequence
$$H^0\Omega^4_{\mathbb{P}^4}(2\bar{X})\longrightarrow H^0(\Omega^4_{\mathbb{P}^4}(2\bar{X})\otimes\mathscr{O}_{P'}\otimes\mathscr{O}_{3P_0})\longrightarrow H^1\Omega^3_{\mathbb{P}^4}(X)\longrightarrow 0$$
\end{lemat}
\begin{proof}[Proof of \ref{CHN}]
By the adjunction formula $K_{X}=(K_{\mathbb{\tilde{P}}^4}+{X})_{|{X}}$ we get $\Omega^3_{X}({X})\cong\Omega^4_{\tilde{\mathbb{P}}^4}(2{X})|X$. This means that we can translate the short exact sequence $$0\longrightarrow\Omega^4_{\tilde{\mathbb{P}}^4}(X)\longrightarrow\Omega^4_{\tilde{\mathbb{P}}^4}(2{X})\longrightarrow\Omega^4_{\tilde{\mathbb{P}}^4}(2{X})\otimes\mathscr{O}_{{X}}\longrightarrow 0$$ into 
$$0\longrightarrow\Omega^4_{\tilde{\mathbb{P}}^4}(X)\longrightarrow\Omega^4_{\tilde{\mathbb{P}}^4}(2{X})\longrightarrow\Omega^3_{X}({X})\longrightarrow 0.$$
From that we obtain the long exact sequence with
 $$H^1\Omega^4_{\tilde{\mathbb{P}}^4}(X)\longrightarrow H^1(\Omega^4_{\tilde{\mathbb{P}}^4}(2{X}))\longrightarrow H^1\Omega^3_X({X})\longrightarrow H^2\Omega^4_{\tilde{\mathbb{P}}^4}(X)=0$$ and so $$ H^1\Omega^3_X({X})\cong  H^1(\Omega^4_{\tilde{\mathbb{P}}^4}(2X)\cong H^1\Omega^4_{\mathbb{P}^4}((2\bar{X})\otimes\mathscr{J}_{P'+3 P_0}).$$ using (6) from Proposition 1 for the last isomorphism.
As we have the exact sequence 
$$0\longrightarrow \Omega_{\mathbb{P}^4}^4(2X)\otimes\mathscr{J}_{P'+3 P_0}\longrightarrow\Omega_{\mathbb{P}^4}^4(2\bar{X})\longrightarrow\Omega_{\mathbb{P}^4}^4(2\bar{X})\otimes\mathscr{O}_{P'+3 P_0}\longrightarrow 0$$
the derived one finishes the proof.
\end{proof}

\begin{proof}
In the following commutative diagram, by the preceeding lemmas we have exact rows.
\begin{tikzcd}
S^{\oplus5}/\mathbb{C} \arrow[d, "\cong", shift right]                            & \mathbb{C}^{4} \arrow[d, "\cong", shift left]                                                   &                                                           &   \\
H^{0}(\Omega^3_{\mathbb{P}^4}(\bar{X})) \arrow[d, "\xi"] \arrow[r, "\theta"] & H^0(\Omega_{\mathbb{P}^4}(\bar{X})\otimes\mathscr{O}_{P_0}) \arrow[d, "\beta"] \arrow[r, "\alpha"]           & H^1(\Omega^3_{\tilde{\mathbb{P}}^4}(X)) \arrow[d, "\phi"] \arrow[r] & 0 \\
H^0(\Omega^4_{\mathbb{P}^4}(2\bar{X})) \arrow[r, "\eta"]                     & H^0(\Omega_{\mathbb{P}^4}(2\bar{X})\otimes\mathscr{O}_{P'}\otimes\mathscr{O}_{3P_0}) \arrow[r, "\gamma"] & H^1(\Omega^3_{X}(X)) \arrow[r]             & 0 \\
S^5 \arrow[u, "\cong"]                                                 & \mathbb{C}^{15+\mu_2} \arrow[u, "\cong"]                                                                  &                                                           &  
\end{tikzcd}
\\
Again, we adapt the proof from \cite{CS3}. Let $K_i$ be the contraction with the vector field $\frac{\partial}{\partial \bar{X}_i}$ and $\Omega=\Sigma_{i=0}^{4}(-1)^id\bar{X}_0\wedge\dots\wedge d\Hat{\bar{X}}_i\wedge\dots\wedge \bar{X}_4$.
The isomorphisms in the first column are described as follows. Linear polynomials $(A_0,\dots, A_4)\in S^{\oplus 5}$ map to $\Sigma _{i=0}^4 \frac{A_i}{F}K_i\Omega\in H^0(\Omega_{\mathbb{P}^4}^3(\bar{X}))$ and degree $5$ homogeneous polynomial $A\in S^5$ goes to $\frac{A}{F^2}\Omega\in H^0(\Omega_{\mathbb{P}^4}^4(2\bar{X}))$. The homomorphism $\eta$ associates to polynomial $A$ its 3-jets at point $P_0$ and evaluates it at points $P_1,\dots, P_\mu$. Similarily $\theta$ assigns value at $P_0$ to quintuples of linear polynomials. 
The morphism $\xi$ is the exterior derivative 
$$d(\Sigma_{i=0}^{4}\frac{A_i}{F}K_i\Omega)=\frac{1}{F^2}\Sigma_{i=0}^{4}(F\frac{\partial A_i}{\partial \bar{X}_i}-A_i\frac{\partial F}{\partial \bar{X}_i})\Omega$$
indeed identifying quintuplet of degree $1$ polynomials with one degree $5$ polynomial

We are to calculate $\dim(\text{Ker}\Phi)$. Note that $h^1(\Omega^3_X(X))=h^1(\Omega^3_{\tilde{\mathbb{P}}^4}(X))-\dim \text{Ker}\Phi + \dim\text{Coker}\Phi$. On the other hand for any $b\in H^0(\Omega_{\mathbb{P}^4}(\bar{X})\otimes\mathscr{O}_{P_0})$ we see that $(\Phi\circ\alpha)(b)=0\iff(\gamma\circ\beta)(b)=0$. As $\beta$ is injective we can, abusing notation, write $b\in H^0(\Omega_{\mathbb{P}^4}(2\bar{X})\otimes\mathscr{O}_{P'}\otimes\mathscr{O}_{3P_0})$. But then $\gamma(b)=0\iff b\in\text{Im}(\eta)$. From that $h^1(\Omega^3_X(X))=\dim \text{Im}(\beta )-\dim (\text{Im}(\beta )\cap \text{Im}(\eta))+\dim\text{Coker}(\beta\circ\gamma)$. But $\dim\text{Coker}\Phi=\dim\text{Coker}(\beta\circ\gamma)$ because it is the dimension of the component of $H^1(\Omega^3_X(X))$ that is exactly not in the image of $\Phi$ and since $\alpha$ is surjective and $\beta$ injective this is the same as the component not in the image of $\beta\circ\gamma$. Putting the equalities together we obtain $\dim(\text{Ker}\Phi)=h^1(\Omega^3_{\tilde{\mathbb{P}}^4}(X))-\dim(\text{Im}\beta)+\dim(\text{Im}(\beta)\cap\text{Im}(\eta))$
By \cite{CS3} Lemma 2. we have the following formula for $h^{1,2}(X)$:
$$h^{1,2}(X)=h^0(\Omega^3_{X}(X))-h^0(\Omega^3_{\tilde{\mathbb{P}^4}}(X))+\dim \text{Ker}(H^1\Omega^3_{\tilde{\mathbb{P}^4}}(X)\to H^1\Omega^3_{X}(X))$$
and putting it together with the above result we get 
$$h^{1,2}(X)=h^0(\Omega^3_{X}(X))-h^0(\Omega^3_{\tilde{\mathbb{P}^4}}(X))+h^1(\Omega^3_{\tilde{\mathbb{P}}^4}(X))-\dim(\text{Im}\beta)+\dim(\text{Im}(\beta)\cap\text{Im}(\eta))$$

From the discussion above we obtain $\text{Im}(\beta)\cap\text{Im}(\eta)=(I_{eq}/ (\bigcap_{i=1}^{\mu}m_i\cap m_0^3))^{(5)}$. Then we calculate
\begin{equation*}
\begin{split}
 h^{1,2}(X)=   & h^0(\Omega^3_{X}(X))-h^0(\Omega^3_{\tilde{\mathbb{P}^4}}(X))+h^1(\Omega^3_{\tilde{\mathbb{P}}^4}(X))-4+\dim(I_{eq}/ (\bigcap_{i=1}^{\mu}m_i\cap m_0^3))^{(5)})=  \\
& h^0(\Omega_{\mathbb{P}^4}^4(2\bar{X})\otimes\mathscr{J}_{3P_0+P'})-h^0(\Omega^4_{\bar{\mathbb{P}}^4}(X))-h^0(\Omega^3_{{\mathbb{P}}^4}(\tilde{X})\otimes\mathscr{J}_{P_0})+ \\ 
& h^1(\Omega^3_{{\mathbb{P}}^4}(\tilde{X})\otimes\mathscr{J}_{P_0})-4+\dim(I_{eq}/ (\bigcap_{i=1}^{\mu}m_i\cap m_0^3)^{(5)})= \\
& \dim((\bigcap_{i=1}^{\mu}m_i\cap m_0^3)^{(5)})-h^0(\Omega^4_{\bar{\mathbb{P}}^4}(X))+ \\
& h^0(\Omega^3_{\mathbb{P}^4}(\tilde{X})\otimes\mathscr{O}_{P_0})-h^0(\Omega^3_{\mathbb{P}^4}(\tilde{X}))-4+\dim(I_{eq}/ (\bigcap_{i=1}^{\mu}m_i\cap m_0^3)^{(5)})= \\ 
&\dim(I_{eq})^{(5)}-{{d-1}\choose{4}}+4-5{ {d}\choose{4}}+{{d-1}\choose{4}}-4=\\ 
&\dim(I_{eq})^{(5)}-5{ {d}\choose{4}}= {{d-1}\choose{4}} -5{ {d}\choose{4}}-11\mu_{3} -\mu_{2}+\delta.
\end{split}
\end{equation*}
Again following \cite{CS4} we remark that $e(X)=e(\bar{X})+2\mu_2+24\mu_3$ as the Milnor number at an ordinary triple point is 16 and the resolution replaces it with a cubic surface with Euler number 9; also the small resolution of a double point with Milnor number $1$ replaces it with a $\mathbb{P}^1$ with Euler number $2$. Thus, using the fact that $e(X)=2(h^{1,1}(X)-h^{1,2}(X))$ we obtain $h^{1,1}(X)=h^{1,1}(\bar{X})-h^{1,2}(\bar{X})+h^{1,2}(X)+\mu_2+12\mu_3= \dim(I_{eq})^{(5)}-{{2d-1}\choose{4}}+12\mu_3+1=1+\mu_3+\delta$.
\end{proof}

\subsection{Hodge numbers of $\tilde{Y}$} 

We recall a definition from \cite{MR}.
\begin{defi}
A process $T(X,Y,\tilde{Y})$ of going from $X$ to $\tilde{Y}$ where $\pi:X\to Y$ is a birational contraction of a Calabi-Yau $X$ onto a normal variety $Y$ and $\tilde{Y}$ is a smoothing of $Y$ is called a geometric transition. A transition is called conifold if $Y$ admits only ordinary double points as singularities. 
\end{defi}

In what follows we distinguish a curve $C$ over which the surface $E$ is ruled and the curve $\tilde{C}$ to which it is being contracted. The curves are birational and in particular $g(C)=g(\tilde{C})$. We use the Theorem 1.3 from\cite{MG2}:
\begin{thm}
Let $\pi:X\to Y$ be a primitive type III contraction of a non-singular Calabi-Yau threefold X contracting a divisor E to a curve $\tilde{C}$. If $g(\tilde{C})\geq 1$ then Y is smoothable.
\end{thm}

\begin{proof}
By \cite[Proposition 1.2]{MG2} and \cite[Proposition 6.5]{YN}  the image of the map $Def(f)\to Def(X)$ has codimension $\geq g(\tilde{C})$. Image of this map is exactly the locus in $Def(X)$ where $E$ deforms in the family of $\mathscr{X}$. Thus, whenever $g(\tilde{C})\geq 1$, then for general $t\in\Delta$ the exceptional divisor $E$ does not deform to ${X}_t$. We consider separately the case when $g(\tilde{C})>1$ and $g(\tilde{C})=1$. 
When $\mathscr{X}\to\Delta$ is a Kuranishi family for $X=X_0$ such that $X$ contains a smooth surface $E$ ruled over a curve $C$ of genus $g>1$ then for every fibre $Z$ of $E$ and some disc around $0$ we have that every $X_t$ in the family contains a curve which is a deformation of $Z$. More specifically, following  discussion in section 4 of \cite{PMHW1}, fibers over $2g-2$ points of $C$ will extend sideways in the family of deformations $\mathscr{X}$. On the other hand the locus $\Gamma$ for which $E$ deforms in the family is of codimension $g$. From that we conclude that whenever $g>1$ the primitive contraction of type III $\pi:X_0\to Y_0$ deforms to family of type I contractions $\pi_t:X_t\to Y_t$ of $2g-2$ fibers. This means that for general $t$ the contraction ${X}_t\to{Y}_t$ is a small (type I) primitive contraction and unless $\text{Sing}({Y}_t)$ consists of exactly one ordinary double point then ${Y}_t$ is smoothable by \cite[Proposition 5.1]{MG1}. As $E$ is normal $Y_t$ cannot have only one ODP so general $Y_t$ is smoothable. 
In case when $g({C})=1$ we have that neither $E$ nor $L$ deforms in the family $\mathscr{X}$ and $X_t\to Y_t$ is an isomorphism for general $t$ and so $Y$ is smoothable. 
\end{proof}
\begin{prop}
Let $X$ be a smooth Calabi-Yau threefold containing a smooth surface $E$ ruled over a smooth curve $C$ of genus $p_a(C)>1$, let $\pi:X\to Y$ be a primitive type III contraction and let $\tilde{Y}$ be the smooth Calabi-Yau obtained by deforming $Y$. Then
\begin{enumerate}[resume]
\item $h^{1,1}(\tilde{Y})=h^{1,1}(X)-1$
\item $h^{1,2}(\tilde{Y})=h^{1,2}(X)+2p_a(C)-3.$
\end{enumerate}
\end{prop}
\begin{proof}
Let $\pi:\mathscr{X}\to \Delta$ be the Kuranishi family for $X=X_0$ where $\Delta$ is a polydisc in $H^1(T_X)$ which we identify with the space of deformations of $X$. By the above theorem the primitive type III contraction $\pi_0:X\to Y$ deforms to a primitive type I contraction $\pi_t:X_t\to Y_t$ of $2p_a(C)-2$ curves of $X_t$. As $E$ is a smooth surface ruled over a smooth curve $C$ we see from the discussion in the proof of Proposition 4.2 of \cite{PMHW1} that the fibers are being contracted to $A_1$ singularities and thus indeed this process yields a conifold transition $T(X_t,Y_t,\tilde{Y})$. Note that we have used the assumption that $E$ is a ruled surface over a smooth curve with $g(C)>1$. In general a type III contraction deforming to type I contractions may yield some worse singularities, for example if $E$ has double fibers.


By \cite[p. 562]{PMHW1},  we can identify groups $H^2(X_t,\mathbb{Z})\cong H^2(X,\mathbb{Z})$ in a family $\pi:\mathscr{X}\to \Delta$ for some polydisc $\Delta\in H^1(T_X)$. We perfom the contractions $\pi_t:X_t\to Y_t$ over $\Delta$ and then the smoothing of the image of the type I contraction. From \cite[Proposition 3.1]{MG1} and \cite[(12.2.1.4.2)]{KJ2} we know the Picard number of $Y_t$ is constant for $t\in\Delta$ (shrinking $\Delta$ if necessary) and thus $h^{1,1}(\tilde{Y})=h^{1,1}(Y_t)=h^{1,1}(X_t)-1=h^{1,1}(X)-1$. 
The final equality follows easily from \cite{MR} Theorem 3.2. It gives that for a conifold transiton $T(X_t,Y,\tilde{Y})$ we have $h^{1,1}(\tilde{Y})=h^{1,1}(X_t)-k$,  $h^{1,2}(\tilde{Y})=h^{1,2}(X_t)+c$ where $k+c=|\text{Sing}(Y)|$. Since $|\text{Sing}(Y)|=2p_a(C)-2$ we obtain $c=2p_a(C)-3$ as $k=1$ concluding the proof.
\end{proof}
\section{Examples of Calabi-Yau threefolds from type III contractions}
\subsection{Construction}
We finish the paper by providing explicit construction of a family of Calabi-Yau threefolds arising as the examples of the procedure described above. 
Recall that we define $X=V(F)$ where $F=u^2F_3+uF_4+F_5$ is a homogenous polynomial of degree $5$ with $V(F_i)$'s being smooth degree $i$ surfaces in $\mathbb{P}^3$ containing a smooth curve $C$. As our construction requires this curve to be simultaneously contained in degree $3,4$ and $5$ surfaces we limit ourselves to the curves of degree $\leq 15$. Choosing a higher degree curve would mean that $V(F_3)\subset V(F_4)\cap V(F_5)$ and so in particular $F_5=F_2F_3$ which leads to $u=F_2=F_3=0$ being a singular curve of $V(F)=V(u^2F_3+uF_1F_3+F_2F_3)$. In choosing the particular curves we follow the classification of possible genus-degree combinations of curves in $\mathbb{P}^3$ in \cite{GP}. We do not claim this list is complete. We construct the curves as ones living on a cubic surface in $\mathbb{P}^3$ either by intersecting it with some other surface or by pulling back a curve on $\mathbb{P}^2$ under the blow up of 6 points in general position as for example in \cite[V 4.7]{RH}. In Table \ref{tabcur} we provide the list of curves that serve as a basis for our construction. The table gives the curve coordinates in Picard group of a cubic surface as well as its genus $g(C)$ and canonical class $K_C$. In general by $C_i$ we mean a curve of degree $i$. Curve $C_{a,b}$ is obtained as an intersection of degree $a,b$ surfaces in $\mathbb{P}^3$ and TC is a twisted cubic. Whenever we discuss more than one curve of a given degree we add a capital letter to the index to distinguish them. When we refer to the specific quintic $\bar{X}$ (or its resolution $X$) containing a given curve $C_i$ we denote it by $\bar{X}_i$ (resp $X_i$). Same applies for $Y_i$ and $\tilde{Y}_i$.

\begin{prop}
Threefold $\bar{X}$ contains a cone over a curve $C$ and the vertex of the cone $O$ is a triple point of $\bar{X}$
\end{prop}

\begin{proof}
Since $F_i(C)=0$ it follows that $F(C)=0$. As $C$ is defined in $\mathbb{P}^3$ it is independent of $u$ and thus it defines a cone in $\bar{X}\subset \mathbb{P}^4 $. The vertex of the cone over $C$ is $O=[0:0:0:0:1]$. Locally around $O$ the equation of $\bar{X}$ is just $F_3$ and this defines a smooth cubic surface thus $O$ is an ordinary triple point of $\bar{X}$. 
\end{proof}

\begin{fact}
For nearly all quintic threefolds $\bar{X}$ we have found the intersection of $V(F_3), V(F_4)$  and $V(F_5)$ consists of the curve $C$ and a finite number of points (we call them excess points). Thus, these quintic threefolds contain not only a cone $\bar{E}$ but also a set of lines all passing through the vertex $O$. The exceptions are quintics containing curves $C_{2A}, C_{10B}, C_{11}, C_{12}$ and $C_{15}$ as for them there are no excess points and thus no excess lines. 
\end{fact}

Before me move further with the discussion we want to discuss one more interesting property of quintic threefolds we work with.
\begin{fact}
We can look at the geometry of quintic $\bar{X}$ as that of a double octic that is a double covering of $\mathbb{P}^3$ branched over an octic surface (e.g. as in \cite{CSZ}). Recall that $\bar{X}$ is described as a zero locus of  $F(x,y,z,t,u)=u^2F_3(x,y,z,t)+uF_4(x,y,z,t)+F_5(x,y,z,t)$. We can project $X$ onto $\mathbb{P}^3$ from the point $[0:0:0:0:1]$. Then the branching locus is defined as the vanishing set of the discriminant, namely a surface of degree $8$, necessarily of the form $S=V(F_4^2-4F_3F_5)\subset\mathbb{P}^3$.
\end{fact}
There  are many interesting properties linking the geometry of $X$ with that of $S$. We show the basic one.
\begin{prop}
Singular points of $S$ contain intersection of $V(F_3)$, $V(F_4)$, $V(F_5)$. In particular, $C\subset Sing(S)$.
\end{prop}
\begin{proof}
Let $G:=F_4^2-4F_3F_5$.
Straightforward calculation of derivatives gives us $\frac{\partial G}{\partial \alpha}=2F_4\frac{\partial F_4}{\partial \alpha}-4(F_3\frac{\partial F_5}{\partial \alpha}+F_5\frac{\partial F_3}{\partial \alpha}$ for $\alpha\in\{x,y,z,t\}$ and so whenever $F_3=F_4=F_5=0$ we have $\frac{\partial G}{\partial \alpha}=0$.
This shows that singular locus of $S$ contains curve $C$ as well as the excess points they are all contained in the intersection of $X_3, X_4, X_5$. 
\end{proof}

We will not discuss the properties of the double cover any further as this is not the aim of our paper. It is worth noting though that some restrictions on the existence of singular octic surfaces translate into restrictions of existence of quintic threefolds with triple point(s). For example one can show that for $m>1$ quintic threefold with $m$ ordinary triple points is a double cover of $\mathbb{P}^3$ branched over an octic surface with $m-1$ quadruple points. As there are restrictions to the number of isolated quadruple points on an octic \cite[Proposition 5.1]{CS},  this provides interesting results. In general, quintic threefold containing a triple point must contain $60$ lines, as intersection of $V(F_3), V(F_4)$ and $V(F_5)$ consists of $60$ points. Yet, for $m\geq10$ a quintic threefold must necessarily contain cones whose vertices are triple points in question as singularites on octic can no longer be isolated forcing $V(F_3), V(F_4)$ and $V(F_5)$ to intersect in a (non necessarily irreducible) curve. This may prove fruitful in further research.

Now we want to describe the K\"{a}hler cone of the constructed threefolds. We use the formula for $h^{1,1}(X)$ from Section 3 to calculate the defect $\delta$ in Macaulay2. As expected we obtain $\delta=1$ for all quintics ${X}_i$ and thus $h^{1,1}(X_i)=3$ which is also the rank of their Picard groups as discussed previously. The following Theorem is a straightforward application of Theorem \ref{KC}.
\begin{thm}
For all quintic threefolds we have obtained, the closure $\bar{K}$ of the K\"{a}hler cone $K$ is a convex hull of three rays. Furthermore, two of the rays are generated by the divisors $H$ and $H-D$ and don't lie on the cubic cone $W$. The third ray is generated by a divisor $L$ such that $L^3\neq 0$ for all quintics but for $X_{2,3}, X_{8B}, X_{11}, X_{3,5}$ for which $L^3=0$.
\end{thm}
In Table \ref{tabdivl} we present values $\beta$ for which $Q_{|D}$ stops being ample, the divisor $L=(2-\beta)H+\beta D+E$ and value $L^3$ for each of the quintics $X$. In column $\perp L$ we list curves $c$ on $D$ for which $c.L=0$. By $C$ we mean the section of $E$, namely $E.D$, $e_i$ are six exceptional divisors of the blow-up of $\mathbb{P}^2$, $g_i=2h+\Sigma_{j\neq i} e_j$, $f_{i,j}=h+e_i+e_j$ are the remaining lines. We write $D$ when $L$ has zero intersection with the whole cubic surface. 

We are now ready to prove the following statement:
\begin{thm}
Quintic threefolds $X_i$ constructed in this paper admit 
\begin{enumerate}
\item a type I, a type II, and a type III primitive contractions, for $i\in\{3, \{2,3\}, \{3,3\}, 9B, \{3,4\}, \{3,5\}\}$
\item two type I primitive contractions and a type III primitive contraction in other cases.
\end{enumerate}
\end{thm}
\begin{proof}
This theorem is a consequence of the previous one and of the Fact 1 in \cite{PMHW1}. Codimension one face spanned by $H-D$ and $H$ corresponds to a type I contraction as it contracts $r$, face spanned by $H-D$ and $L$ corresponds to a type III contraction as it contracts all fibers $l$ of $E$ and thus contracts $E$ to a curve. The face spanned by $L$ and $H$ corresponds to either a type II contraction when $L.D=0$ or to a type I contraction in the remaining cases.  
\end{proof}
Note the fact that $L^3=0$ for some of the quintics does not affect the contraction corresponding to the faces of the cone limited by $L$ as for all divisors $Q$ on these faces $Q^3\neq 0$. Also take notice that the contraction of $X$ given by the system $H$ is exactly the blow-up of a triple point and the curves coming from the small resolution, while system $H-D$, that is system of hyperplane sections passing through $D$, which translates to hyperplane sections passing through the point $O$ on $\bar{X}$ is the system giving a double cover of $\mathbb{P}^3$ discussed earlier. 
In the following section we take a closer look at the type III contraction.

\subsection{Contraction and smoothing}
We restrict our focus to the type III contraction admitted by quintics $X_i$. 
\subsection{Linear system}
We keep the notation as above. The goal of this section is to prove three theorems regarding the existence of type III contractions. 
\begin{thm}
For discussed quintics $X_i$, where $i\notin\{1, 2,\{TC\},\{4A\},\{4B\},\{7A\}\}$, the morphism $\phi_{(|m(2H+E)|)}:X\to Y$ for some $m>>0$ is a primitive contraction of type III which contracts ruled surface $E$ to a curve $\tilde{C}$.
\end{thm}
By $\phi_{(|m(2H+E)|)}$ we mean a morphism given by nef linear system $|m(2H+E)|$ on $X_i$. 
We've excluded those quintics for which $2H+E$ has zero, or negative intersection with some curve on $D$ leaving only those, for which $(2H+E)_{|D}$ (and thus its positive multiple) is ample.

\begin{proof}
First of all we remark that $2H+E$ is indeed contained in the face of the K\"{a}hler cone spanned by $H-D$ and $L$ as it is of the form $(2\gamma-\beta)H+\beta D+\gamma E$ for $\beta=0$ and so $(2H+E).l=0$ for any quintic $X_i$. 
Since $|H|$ is the pullback of a hyperplane section of $\bar{X}$ it separates points on $X$ outside of $D$ and so does $|2H+E|$. Moreover restriction of $T\in|2H+E|$ to $D$ belongs to the linear system $|\tilde{C_i}|$ of curves linearly equivalent to some curve $\tilde{C}_i$ on $D$ (different for each $X_i$). For all quintics we discuss here the system $|\tilde{C_i}|$ is (very) ample yet we are not sure $|2H+E|$ maps surjectively onto it and thus whether or not its restriction to $D$ separates points. Yet, following the Fact 1 of \cite{PMHW1} we know codimension one faces of the K\"{a}hler cone correspond to the primitive contractions thus some multiple of $|2H+E|$ gives the desired morphism finishing the proof.
\end{proof}

Note $mT\in |m(2H+E)|$ corresponds to the hyperplane section of $Y$ and thus we have that $mT.E$ is actually the degree of the curve $\tilde{C}$. By similar argument we obtain the degree of $Y$ namely  
$(mT)^3=m^3(2H+E)^3=m^3(36+6\deg(C)+4g(C)).$
We have strong reasons to believe that $m=1$ should suffice to give the desired morphism although we haven't been able to show it yet. 

We need to deal with the remaining threefolds and their respective type III contractions.
\begin{thm}
For quintics $X_i$ where $i\in\{2,\{TC\},\{4A\},\{4B\},\{7A\}\}$ the morphism $\phi_{(m|3H-D+E|)}:X\mapsto Y$, for some $m>>0$ is a primitive contraction of type III which contracts ruled surface $E$ to a curve $\tilde{C}$.
\end{thm}
\begin{proof}
We can follow the proof of the previous theorem. As in the preceeding case we note that $T:=3H-D+E$ lies on the face spanned by $H-D$ and $L=2H+E$. Intersection of $T$ with any curve not contained in $D$ or $E$ is positive as $3H$ intersects any such curve. The divisor is chosen so it has zero intersection with fiber $l$ of $E$ and in all cases of this theorem $(-D+E)_{|D}$ is an ample curve on $D$ thus intersection with any curve contained in $D$ is positive. Again by Fact 1 of \cite{PMHW1} we obtain the desired morphism.

\end{proof}
\begin{thm}
For quintic $X_1$  the morphism $\phi_{(m|4H-2D+E|)}:X\mapsto Y$ for some $m>>0$ is a primitive contraction of type III which contracts ruled surface $E$ to a curve $\tilde{C}$.
\end{thm}
\begin{proof}
The proof follows the same argumentation as the two before. 
\end{proof}
Similarily to the case where $L=m(2H+E)$ we obtain the degrees of $Y$'s and $\tilde{C}$'s. Table \ref{tabhod} summarizes the obtained results.

This brings us to the final result of this section.
\begin{cor}
The quintics $X_i$ for $i\in\{3,4A, 4B, 5A, 5B, 6,\{2,3\}, 7A, 7B, 7C, 8A, 8B,\\ 9A, 9B,  \{3,3\}, 10A, 10B, 11, \{3,4\}, \{3,5\}\}$ admit a primitive, type III contraction to a threefold $Y_i$  which is smoothable thus yielding a smooth Calabi-Yau threefold $\tilde{Y}_i$.
\end{cor}
\begin{proof}
Immediate from the theorem above and discussion from the previous section.
\end{proof}

Table \ref{tabhod} contains the calculated Hodge numbers for $\tilde{Y}_i$ obtained from smoothing of $Y_i$ in cases $g(\tilde{C})>1$

\clearpage
\begin{center}
\begin{table}
\centering
\caption{Intersection of curves and divisors on $X$}
\label{tabint}
\begin{tabular}{|c|c|c|c|}
\hline
       & $H_X$ & $E$ & $D$   \\ \hline
l      & 1      & -2   & 1     \\ \hline
t     & 1      & 0     & 1     \\ \hline
r     & 0      & 1     & 0     \\ \hline
h      & $0  $   & $a$   &$ -3$\\ \hline
$e_i $  &$ 0  $    &$ b_i  $   &$-1$     \\ \hline
C      & $0$      & $a^2-\Sigma b_i^2$    & $-3a+\Sigma b_i$  \\ \hline
\end{tabular}
\end{table}
\end{center}

\begin{table}[]
\centering
\caption{Curves at a basis of a cone in Calabi-Yau threefolds}
\begin{tabular}{|l|l|l|l|l|l|l|l|l|l|}
\hline
$C$ 	&  $g(C)$ & $ K_C$ 	& a&    $b_1$ & $b_2$ & $b_3$ & $b_4$ & $b_5$ & $b_6$  \\ \hline
$C_1$    	& 0     	&  -2	& 1 &    $1$ & $1$ & $0$ & $0$ & $0$ & $0$           \\ \hline
$C_2$  	& 0     &  -2 	& 1  &   $1$ & $0$ & $0$ & $0$ & $0$ & $0$     \\ \hline
TC   		& 0     	&  -2 	& 1   &  $0$ & $0$ & $0$ & $0$ & $0$ & $0$        \\ \hline
$C_3$    	& 1    &  0  	&  3   &  $1$ & $1$ & $1$ & $1$ & $1$ & $1$      \\ \hline
$C_{4A}$  	& 1     &  0 	&4 &     $2$ & $2$ & $1$ & $1$ & $1$ & $1$      \\ \hline
$C_{4B}$      & 0    &  -2 	& 4 &    $3$ & $1$ & $1$ & $1$ & $1$ & $1$      \\ \hline
$C_{5A}$     	& 2    &  2  	&5   &   $2$ & $2$ & $2$ & $2$ & $1$ & $1$ \\ \hline
$C_{5B}$  	& 2    &  2  	&4    &  $2$ & $1$ & $1$ & $1$ & $1$ & $1$    \\ \hline
$C_6$   	& 4    &  6  	&7    &  $3$ & $3$ & $3$ & $2$ & $2$ & $2$  \\ \hline
$C_{2,3}$    	& 3     &  4 	&6    &  $2$ & $2$ & $2$ & $2$ & $2$ & $2$  \\ \hline
$C_{7A}$     	& 1    &  0  	&3    &  $1$ & $1$ & $0$ & $0$ & $0$ & $0$   \\ \hline
$C_{7B}$    	& 5     & 8	 	& 7    & $3$ & $3$ & $2$ & $2$ & $2$ & $2$   \\ \hline
$C_{7C}$  	& 4    &  6	 	&7&$4$ & $2$ & $2$ & $2$ & $2$ & $2$   \\ \hline
$C_{8A}$   	& 6   &  10   	&8    &  $3$ & $3$ & $3$ & $3$ & $3$ & $1$   \\ \hline
$C_{8B}$      & 7     &  12	 &7  &   $3$ & $2$ & $2$ & $2$ & $2$ & $2$    \\ \hline
$C_{9A}$  	& 10    &  18	 &9   &  $4$ & $4$ & $3$ & $3$ & $2$ & $2$       \\ \hline
$C_{9B}$   	& 9     &  16		&11  &  $4$ & $4$ & $4$ & $4$ & $4$ & $4$  \\ \hline
$C_{3,3}$ 	& 8    	&  14	&9   &  $3$ & $3$ & $3$ & $3$ & $3$ &$3$    \\ \hline
$C_{10A}$   	& 12   &  22 	 &11 &  $4$ & $4$ & $4$ & $4$ & $4$ & $3$ \\ \hline
$C_{10B}$    & 11  &  20   	& 10  & $4$ & $4$ & $4$ & $3$ & $3$ & $2$    \\ \hline
$C_{11}$  	& 15   &  28  	&11  &  $4$ & $4$ & $4$ & $4$ & $3$ & $3$     \\ \hline
$C_{3,4}$    	& 19   &  36 	 &12&   $4$ & $4$ & $4$ & $4$ & $4$ & $4$ \\ \hline
$C_{3,5}$    	& 31   &  60 	 &15 &  $5$ & $5$ & $5$ & $5$ & $5$ & $5$    \\ \hline
\end{tabular}
\label{tabcur}
\end{table}

\begin{table}[]
\centering
\caption{Divisor L on X}
\begin{tabular}{|l|l|l|l|l|}
\hline
X & $\beta$ & L & $L^3$ &$\perp L$ \\ \hline
$X_1$   	& -1      & 3H-D+E    &124 & $C, f_{1,2}$ \\ \hline
$X_2$		& 0    	  & 2H+E    	&48	 & $C, e_2,\dots, e_6, f_{1,2},\dots, f_{1,6}$ \\ \hline
$X_{TC}$    	& 0      & 2H+E    	&54    & $e_1,\dots, e_6$     \\ \hline
$X_3$   	& 1      & H+D+E     	& 8	& $D$    \\ \hline
$X_{4A}$   	& 0      & 2H+E    	& 64 	& $f_{1,2}$      \\ \hline
$X_{4B}$    	& 0      & 2H+E    	& 60	& $f_{1,2},\dots, f_{1,6}$    \\ \hline
$X_{5A}$     	& 1      &H+D+E  	 & 12  & $e_5, e_6, g_5, g_6, f_{1,2},f_{1,3},f_{2,3},f_{2,4},f_{3,4}$     \\ \hline
$X_{5B}$     	& 1      & H+D+E   	 & 12  & $e_2,\dots,e_6, f_{1,2},\dots, f_{1,6}$     \\ \hline
$X_{2,3}$     & 2      &2D+E    	 & 0 	& $D$     \\ \hline
$X_6$     	& 1      &H+D+E   	 & 16	& $g_5, g_6, f_{1,2}, f_{1,3}, f_{2,3}$   \\ \hline
$X_{7A}$     	& 0      &2H+E      	& 82	& $e_2,\dots, e_6$       \\ \hline
$X_{7B}$     	& 1      &H+D+E   	& 24  	& $f_{1,2}$     \\ \hline
$X_{7C}$     	& 1      & H+D+E   	&20	& $ f_{1,2},\dots, f_{1,6}$     \\ \hline
$X_{8A}$     	& 1      &H+D+E    	&28 	& $g_1, g_6$     \\ \hline
$X_{8B}$     	& 2      & 2D+E     	&0 	& $e_2,\dots, e_5, g_4, f_{1,6}, f_{2,6}, f_{3,6}, f_{4,6}, f_{5,6}$       \\ \hline
$X_{3,3}$     & 3     &-H+3D+E 	& 4   	& $D$         \\ \hline
$X_{9A}$     	& 2      &2D+E      	& 2  	& $g_1,\dots, g_6$          \\ \hline
$X_{9B}$     	& 1      &H+D+E       & 36	& $f_{1,2}$        \\ \hline
$X_{10A}$    & 2     & 2D+E   	& 8 	& $g_6$        \\ \hline
$X_{10B}$    & 2     &2D+E 	& 4  	& $e_6, g_6, f_{1,2}, f_{1,3}, f_{2,3}$          \\ \hline
$X_{11}$   	& 3     &-H+3D+E 	& 0	& $e_5, e_6, g_5, g_6, f_{1,2}, f_{1,3}, f_{1,4},f_{2,3}, f_{2,4}, f_{3,4}$         \\ \hline
$X_{3,4}$   	& 4     &-2H+4D+E  &8 	& $D$            \\ \hline
$X_{3,5}$    	& 5     & -3H+5D+E  & 0  	& $D$          \\ \hline
\end{tabular}
\label{tabdivl}
\end{table}

\begin{table}[]
\centering
\caption{Degree of a variety Y and Hodge numbers of a variety $\tilde{Y}$}
\begin{tabular}{|l|l|l||l|l|l|}
\hline
Y & $\deg(e)$ & $\deg Y$  & $\tilde{Y}$ & $h^{1,1}(\tilde{Y})$ & $h^{1,2}(\tilde{Y})$  \\ \hline
$Y_1$   	& 1m      & 310$m^3$ &-&-&-  \\ \hline
$Y_2$   	& 2m      & 132$m^3$&-&-&- \\ \hline
$Y_{TC}$   	& 4m      & 144$m^3$  &-&-&-   \\ \hline
$Y_3$   	& 3m     & 58$m^3$  &    	&     &    \\ \hline
$Y_{4A}$   	& 8m      & 160$m^3$  &-&-&-   \\ \hline
$Y_{4B}$   	& 6m      & 156$m^3$  &-&-&-   \\ \hline
$Y_{5A}$     	& 7m      &74$m^3$ &$\tilde{Y}_{5A}$     	& 2      &55	 	  \\ \hline
$Y_{5B}$     	& 7m     & 74$m^3$  & $\tilde{Y}_{5B}$     	& 2     & 41    \\ \hline
$Y_{2,3}$     & 12m      &88$m^3$ &$\tilde{Y}_{2,3}$     & 2      &57      \\ \hline
$Y_6$     	& 10m      &84$m^3$  &$\tilde{Y}_6$     	& 2     &52   \\ \hline
$Y_{7A}$   	& 14m      & 196$m^3$  &-&-&-   \\ \hline
$Y_{7B}$     	& 15m      &82$m^3$ &$\tilde{Y}_{7B}$     	& 2 &42   	\\ \hline
$Y_{7C}$     	& 13m      &94$m^3$ &$\tilde{Y}_{7C}$     	& 2     & 49	   \\ \hline
$Y_{8A}$     	& 18m      &108$m^3$  &$\tilde{Y}_{8A}$     	& 2   &51           \\ \hline
$Y_{8B}$     	& 20m      & 112$m^3$ &$\tilde{Y}_{8B}$     	&2      & 56     \\ \hline
$Y_{3,3}$     & 27m     &130$m^3$ &$\tilde{Y}_{3,3}$     & 2     &63		   \\ \hline
$Y_{9A}$     	& 25m      &126$m^3$   &$\tilde{Y}_{9A}$     	& 2   &   58         \\ \hline
$Y_{9B}$     	& 23m      &122$m^3$&$\tilde{Y}_{9B}$     	& 2     &53    \\ \hline
$Y_{10A}$    & 32m    & 144$m^3$ &$\tilde{Y}_{10A}$    & 2    &65    \\ \hline
$Y_{10B}$    & 30m     &140$m^3$ &$\tilde{Y}_{10B}$    & 2    &60        \\ \hline
$Y_{11}$   	& 39m     &162$m^3$ &$\tilde{Y}_{11}$   	& 2    &72		     \\ \hline
$Y_{3,4}$   	& 48m     &184$m^3$  &$\tilde{Y}_{3,4}$   	& 2    &84\\ \hline
$Y_{3,5}$    	& 75m     & 250$m^3$ &$\tilde{Y}_{3,5}$    	& 2     &108        \\ \hline
\end{tabular}
\label{tabhod}
\end{table}

\end{document}